%% file: ergshad_main.tex
\documentclass{amsart}

\usepackage{amsfonts,amssymb}
\usepackage[utf8]{inputenc}
\usepackage[english]{babel}
\usepackage{amsthm}
\usepackage{url}

\input{taudefs}

    \title{ Stochastic shadowing and stochastic stability}

    \author{Dmitry Todorov}
    \address{ DMA, \'{E}cole Normale Sup\'{e}rieure, 75005 Paris, France \\
    and
    Chebyshev laboratory, Saint Petersburg State University \\ 14th line of Vasiljevsky Island, 29B \\ 199178, Saint-Petersburg, Russia}

\begin{document}

\maketitle

\input{abstract}
\input{intro}
\input{defs}
\input{mainres}
\input{proof}
\input{ex_tent}

\input{ack}

\input{refs}

%

\end{document}

%% file: taudefs.tex
\newcommand{\Rb}{\mathbb R}
\newcommand{\Nb}{\mathbb N}

\newcommand{\eps}{\varepsilon}

\newcommand{\lb}{\left(}
\newcommand{\rb}{\right)}

\newcommand{\enbrace}[1]{\lb #1 \rb}

\newcommand{\setdef}[2]{\left\{ #1\ \left|\ #2 \right.\right\}}

\newcommand{\seq}[1]{ \{ #1 \} }

\newcommand{\vmod}[1]{\left| #1 \right|}

\newcommand{\Rplus}{\Rb_{\geq 0}}

\newcommand{\normban}[1]{\left\| #1 \right\|}

\numberwithin{equation}{section}

\newcommand{\holder}{H\"{o}lder }

\newcommand{\xbar}{\bar{x}}
\newcommand{\etraj}[1]{#1 }

\newcommand{\etrajfull}[1]{\seq{ f^k(#1)}_{k=0}^\infty }

\newcommand{\fspacel}{ \mathcal{B}_{\mathrm{Lip}} } 
\newcommand{\bspace}{ \mathcal{B} } 
\newcommand{\bspacec}{ \mathcal{B}_c } 
 
\newcommand{\normb}[1]{\normban{#1}_{\bspace} }
\newcommand{\normBV}[1]{\normban{#1}_{BV} }
\newcommand{\normLone}[1]{\normban{#1}_{\Lone} }
\newcommand{\normopKL}[1]{\normban{#1}_{\mathrm{KL}} }

\newcommand{\Lone}{L^1}
\newcommand{\normbl}[1]{\normban{#1}_{\Lone} }
\newcommand{\mfd}{ M }

\newcommand{\distmod}[3]{ \vmod{ \int #1 d #2 - \int #1 d #3  }}  
\newcommand{\differmod}[3]{ \enbrace{ \int #1 d #2 - \int #1 d #3  }}  

\newcommand{\abr}[1]{ \langle #1 \rangle }

\newcommand{\tentparamset}{[\sqrt{2},2]}

\newcommand{\shadparamset}{\mathcal{P}}

\newcommand{\shadname}{stochastic }

\newcommand{\paramset}{\mathcal{A}}
\newcommand{\Xparam}[2]{X^{(#2)}_{#1} }
\newcommand{\mustparam}[2]{\mu^{(#1)}_{#2} }
\newcommand{\muparam}[1]{\mu^{(#1)} }

\newcommand{\epsset}{ (0,\eps_0) }
\newcommand{\obig}[1]{\mathcal{O}\enbrace{#1} }

\newcommand{\seqinf}[2]{\left\{ {#1}_{#2} \right\}_{#2 = 0}^{\infty} }

\newcommand{\Aset}[3]{A^{#1}_{\eps,#2,#3} }
\newcommand{\Bset}[2]{B^{#1}_{\eps,#2} }

\DeclareMathOperator{\dist}{dist}
\DeclareMathOperator{\diam}{diam}

\DeclareMathOperator{\Int}{Int}
\DeclareMathOperator{\var}{var}
\DeclareMathOperator{\Leb}{Leb}

\theoremstyle{plain}
 \newtheorem{theorem}{Theorem}[section]
 \newtheorem{maintheorem}{Theorem}
 
 \newtheorem{statement}{Statement}[section]
 \newtheorem{lm}[statement]{Lemma}
 \newtheorem{sublm}[statement]{Sublemma}
 \newtheorem{corollary}[statement]{Corollary}
 \newtheorem{maincor}{Corollary}
\theoremstyle{definition} \newtheorem{deff}{Definition}
 
 \newtheorem{assume}{Assumption}
\theoremstyle{remark} \newtheorem{rem}[statement]{Remark}

%% file: abstract.tex
\begin{abstract}
    The notion of stochastic shadowing property is introduced.
    Relations to stochastic stability and standard
    shadowing are studied.
    Using tent map as an example it is proved that,
    in contrast to what happens for standard shadowing,
    there are significantly non-uniformly hyperbolic
    systems that satisfy stochastic shadowing property. 
\end{abstract}

%% file: intro.tex
\section{Introduction}

In this paper I address two problems.
The first one is 
getting information about a chaotic dynamical system without 
having access to its exact trajectories
-- i.e. when one has only
trajectories with errors (pseudotrajectories).
The second problem is the lack of examples
of non-uniformly hyperbolic systems having good shadowing
properties.

The most standard naive way to approximate an (SRB)
invariant measure of a chaotic system is to take a point 
of a phase space at random, iterate it (numerically)
for a long time and then build a histogram.
However an issue appears here -- when one does 
numerics, one always gets trajectories with errors
that appear at each iteration.
It is not clear is there a real trajectory close
to each erroneous one and if there is one, does it
reflect well the statistics of the invariant measure one
was that trying to approximate initially.
The same applies of course to experimental observations
of chaotic physical systems.

There are two well-known important notions in this context:
stochastic stability and shadowing.
Both notions state a sort of stability of the system with
respect to small per-iteration perturbations.
So far no direct relations between them were established 
(it was stated as a problem
in \cite{BONATTI_DIAZ_VIANA__BEYOND}).

Both notions have certain drawbacks when one tries to
apply them to a practical situation.
Namely classical shadowing notions (see \cite{PILSDS,PALSDS}) do not
concern statistical information questions at all, and
stochastic stability works only on a level of measures
and does not take into account the fact that ergodic averages
for the unperturbed system can have fluctuations with different 
properties than ergodic averages for the perturbed one.
That is that usually it is desirable that
one has closeness of statistical properties (of samples given by
erroneous and by exact trajectories)
not only in the limit but also for a (big) finite number
of iterations for a large set of erroneous trajectories.
%


It is well known that smooth uniformly hyperbolic systems have shadowing \cite{PILSDS}.
In fact, they 
have even stronger property of quantitative (Lipschitz)
shadowing. Moreover this
property is uniform with respect to
perturbations in $C^1$ topology.
It was proved that in fact uniformly hyperbolic systems are
the only ones to have Lipschitz and \holder shadowing
(see \cite{PILTIKLSISS, TIKH_HOL_SHAD}).

However even for simplest systems with singularities nothing like this is known. 
For piecewise expanding maps of the interval
there is a simple observation by Blank
\cite{BLANK_EPS_TRAJ},
Lipschitz shadowing away
from singularities by Kifer \cite{KIFER88}
and a shadowing result for a
large set of parameters of tent maps by Coven, Kan and Yorke
\cite{YORKE_TENT_SHAD}.

A term ``ergodic shadowing'' has been introduced by
Fakhari and Ghane in \cite{FAKHARI_ERG_SHAD}. The word ``ergodc'' there
was used not because of relation to ergodic properties
but because of special notion of closeness between
sequences of points.
I use term ``stochastic shadowing''
for a different notion which is more ``ergodic'' 
in some sense.
It is somewhat closer to a notion of average
shadowing introduced by Blank in \cite{BLANK_EPS_TRAJ}.

Average shadowing itself will be too strong for our purposes --
it is difficult to satisfy it for a systems
that do not have very strong hyperbolicity
 (with the ambient
manifold being a hyperbolic set).

We will study relations between stochastic shadowing,
stochastic stability and standard shadowing.
To show usefulness of the new notion we examine
family of tent maps as an example.
We will show that there is no hope to have uniform Lipschitz
shadowing for it but nevertheless
it is possible to show the presence of 
Lipschitz stochastic shadowing.




%% file: defs.tex
\section{Definitions}

Let $M$ be a compact manifold (possibly with boundary)
with a Riemmanian metric $\dist$.
Let $f$ be a mapping of $M$ to itself.

\input{ergdefs_basic}

Consider $\sigma:\Rplus \to \Rplus$ such that
$\sigma(\eps) \to 0$ as $\eps\to 0$.

\begin{deff}
    The map $f$ is said to have \emph{classical shadowing 
        with accuracy $\sigma$
        with respect to $\seq{X_\eps}_{\eps\in\epsset}$ } if  
    for any $0<\eps<\eps_0$
    for every realization of $X_\eps$ there exists a point $p\in M$ such that
    \begin{equation*}
        \dist(x_k, f^k(p) ) < \sigma(\eps), \quad k\geq 0.
    \end{equation*}

    We call $\sigma$ a \emph{classical shadowing accuracy function}.
\end{deff}

\begin{deff}
    The map $f$ is said to have \emph{standard shadowing 
        with accuracy {$\sigma$} }
        if it has classical shadowing 
        with accuracy $\sigma$
        with respect to any admissible family of Markov
        chains.
\end{deff}

\begin{rem}
    If we do not specify accuracy, this definition
    coincides with the definition of
    (one-sided) shadowing property in \cite{PILSDS}.
    If we ask for $\sigma(\eps) = L\eps$ for some $L>0$ then
    the definition coincides with the definition of
    (one-sided) Lipschitz shadowing property
    from \cite{PILSDS}.
\end{rem}

\begin{deff}
    The map $f$ has
    \emph{strong stochastic stability with speed $\sigma$ with 
        respect to $\seq{X_\eps}_{\eps\in\epsset}$ } if 
    both $\mu$ and $\mu_\eps$ have densities (with respect
    to $\Leb$)
    $\rho$ and $\rho_\eps$ respectively and
    for every $0<\eps<\eps_0$ the following estimate holds:
    \begin{gather*}
        \normban {\rho - \rho_\eps}_{\Lone}  \leq \sigma(\eps).
    \end{gather*}

    We call $\sigma$ a \emph{speed of
    strong stochastic stability}.
\end{deff}

\begin{rem}
  Note that here when we do not require an
  accuracy of shadowing to be a function uniquely defined
  by the system and its perturbation. For example
  if a system has classical shadowing with accuracy
  $\gamma(\eps) = L \eps$ for some $L>0$ then
  it has classical shadowing with accuracy
  $\gamma(\eps) = L \eps^\alpha$ for every $0<\alpha<1$.
  
  The same applies for a speed of stochastic stability
  all similar functions (speeds and accuracies)
  we will consider later.
\end{rem}


\input{defs_observables}



We also define one classical notion we will use later.
\begin{deff}
    Let string [a,b] be a finite set of consecutive integers
    $\seq{a, a+1, \ldots , b}$.
\end{deff}

\begin{deff}
    $f$ has specification property if
    for every $\eps>0$ there exists an integer $N(\eps)$
    such that for every choice of points $x_1,x_2\in \mfd$
    and strings $A_1 = [a_1,b_1]$ and $A_2 = [a_2,b_2]$
    with $a_2 - b_1 > N(\eps)$
    and every integer $p>b_2 - a_1 + N(\eps)$
    there exists a periodic point $x\in \mfd$ with period $p$
    such that
    \begin{gather*}
        \dist(f^i(x),f^i(x_1)),\quad i\in A_1;\\
        \dist(f^i(x),f^i(x_2)),\quad i\in A_2.
    \end{gather*}
\end{deff}

\input{defs_uniform}


%
%
%
%

%% file: ergdefs_basic.tex
For $x\in M$ let $\delta_x$ be a $\delta$-measure
concentrated at point $x$.
For a sequence of points $\xi = \seq{x_k}_{k\in J} \subset M$,
where $J$ is either $\Nb\cup\{0\}$ or $\seq{0,1,\ldots,N}$  
 with number of elements greater than $n$
denote 
\begin{equation*}
    S_n(\xi) = \frac{1}{ n+1 } \sum_{k=0}^{ n} \delta_{x_k} .
\end{equation*}
For $x\in M$ denote $S_n(x) = S_n(\etrajfull{x}) $.

Let $\mu$ be a Borel invariant probability measure on $M$.
Invariant means that $\mu(f^{-1}(A)) = \mu(A)$
for every Borel set $A$.

\begin{deff}
    The mesure $\mu$ is \emph{ergodic} if for 
    $\mu$-a.e. point $x\in M$
    \begin{equation}
    S_n(x) \to \mu,\quad n\to \infty    \label{eq:ergconv}
    \end{equation}
    where the convergence holds  in the weak* topology on 
    the space of probability measures on $M$.

    A point $x\in M$ for which \eqref{eq:ergconv} is
    satisfied is called \emph{typical} for $\mu$ 
    or just \emph{$\mu$-typical}.
\end{deff}

\begin{deff}
    The measure $\mu$ is called \emph{physical} if 
    the set of $\mu$-typical points has full (normalized)
    Lebesgue measure
    on $M$.
\end{deff}

\begin{assume}
    Assume $\mu$ is physical for $f$.
\end{assume}

\begin{deff}
    A sequence of points $\xbar = \seq{x_k}_{k=0}^{\infty}$
    is called an $\eps$-pseudotrajectory (of $f$) if
    \begin{equation*}
        \dist(f(x_k),x_{k+1}) \leq \eps,\quad
        k\geq 0.
    \end{equation*}
\end{deff}

Let be $\seq{X_\eps}_{\eps\in \epsset}$ a family
of Markov chains on $M$, where $0<\eps_0 < \diam M /2$.
\begin{deff}
    We call such a family \emph{admissible} if every
    realization of $X_\eps$ is an $\eps$-pseudotrajectory.
\end{deff}

\begin{assume}
    Assume $\seq{X_\eps}_{\eps\in \epsset}$ is admissible.
\end{assume}

For $x\in M$ and a Borel set $E$
let $P_\eps (x,E)$ be a transition probability for the 
chain $X_\eps$ i.e.
\begin{equation*}
    P_\eps(x,E) = P(X_{\eps}^{n+1} \in E | X_{\eps}^n = x),
    \quad n\geq 0.
\end{equation*}

Denote the nomalized Lebesgue measure on $M$ by $\Leb$.
We will sometimes use specific type of perturbations:
\begin{deff}
    We say that the family
    $\seq{X_\eps}_{\eps\in \epsset}$
    is a \emph{family of
    uniform perturbations}
    if for every $0<\eps<\eps_0$
    the transition probability $P_\eps(x,dy)$ 
    for $X_\eps$ is 
    given by a uniform distribution in $\eps$-ball 
    around $f(x)$. 

    I.e. $P_\eps (x,dy)$ 
has the following density for every $x\in M$:
\begin{equation*}
    \enbrace{\frac{d}{\Leb} P_\eps(x,\cdot)}(y) = 
    \frac{1}{\Leb(B_\eps(f(x)) ) }
    \mathbb{\chi}_{B_\eps(f(x)) } (y) ,
\end{equation*}                         
where $\xi_A$ is the indicator function for a set $A$.
\end{deff}


\begin{deff}
    A Borel probability measure on $M$ is said to be
    \emph{stationary for the Markov chain ${X_\eps}$ } if
    for every Borel set $E$
    the following identity holds
    \begin{equation*}
        \mu_\eps (E) = \int P_\eps (x,E) d\mu_\eps(x).
    \end{equation*}
\end{deff}

\begin{deff}
    A measure $\mu_\eps$ stationary for the
    Markov chain $X_\eps$
    is said to be \emph{ergodic} if for $\mu_\eps$-a.e.
    realization $\xbar = \seq{x_k}_{0\leq k \leq \infty}$
    of the Markov chain $X_\eps$
    \begin{equation*}
        S_n(\xbar) \to \mu_\eps,\quad  n\to \infty
    \end{equation*}
    where the convergence holds  in the weak* topology on 
    the space of probability measures on $M$.
\end{deff}

\begin{assume}
    Assume for every $0<\eps<\eps_0$ there is a unique
    ergodic stationary measure $\mu_\eps$ for $X_\eps$.
\end{assume}

Denote the space of continuous functions from $M$ to $\Rb$
with $\sup$-norm
by $(C(M),\normban{\cdot}_\infty)$
and the standard $\Lone$ space with standard $\Lone$ norm
for Borel sigma-algebra by $(\Lone(M),\normbl{\cdot})$.
Let $(\bspace,\normb{\cdot})$ be some Banach space of
functions from $M$ to $\Rb$.
\begin{assume}
    Assume that
    $\bspace \subset (\Lone(M),\normbl{\cdot})$
    and 
    $\bspacec = C(M) \cap \bspace \neq \emptyset$.
\end{assume}


%
%
Let $P^n_\eps$ be the $n$-step transition probability
for the Markov chain $X_\eps$.
\begin{deff}
    We say that 
    the stationary measure
    $\mu_\eps$ \emph{has exponential decay of
        correlations 
    for observables from $\bspace$
} 
        if there exists $\tau>0$ such that for every
    $\tau'>\tau$ for every $\phi,\psi\in \bspace$
        there exists a constant $C = C(\tau,\phi,\psi)$
        such that
        \begin{equation*}
            \vmod{
                \int \enbrace{\int\phi(y) d P^n_\eps(x,y) }
                \psi(x) d \mu(x) -
        \int \phi d\mu_\eps \int \psi d\mu_\eps } 
        \leq
    C(\tau')^n,\quad
        n\in\Nb.
        \end{equation*}
\end{deff}

\begin{deff}
    We say that the stationary measure $\mu_\eps$
    \emph{has property A
    for observables from $\bspace$}
    if for any $\delta>0$, $N\in \Nb$ and any
    $\phi\in \bspace$
    there exists a set $\Aset{\phi}{\delta}{N}$
    of realizations of $X_\eps$
    such that for any $\xbar\in \Aset{\phi}{\delta}{N}$
    for every $n>N$
    the following holds:
    \begin{equation*}
        \distmod{\phi}{S_n(\xbar)}{\mu_\eps} < \delta
    \end{equation*}
    and 
    \begin{equation*}
        \mu_\eps(\Aset{\phi}{\delta}{N} ) \to 1,\quad
        N \to \infty.
    \end{equation*}
\end{deff}

\begin{rem} \label{rem:expdecay_impl_propA}
    If $\mu_\eps$ has exponential decay of correlations
    for observables from $\bspace$,
    then it has property A
    for observables from $\bspace$
    since for every $\phi\in \bspace$ we can write the following estimate
    for some constant $C(\phi)>0$ and some function
    $I(\phi,\delta)>0$:
    \begin{equation*}
        \mu_\eps \enbrace{ \setdef{\xbar}{ \distmod{\phi}{S_n(\xbar) }{\mu_\eps} >
    \delta} } <  C(\phi) e^{-I(\phi,\delta) n} .
    \end{equation*}
    It is left to notice that the sum of the left parts for $n=N\ldots \infty$ 
    is less then $C(\phi,\delta) \exp( -I(\phi,\delta)  N )$
    for some $C(\phi,\delta)>0$.
\end{rem}

%% file: defs_observables.tex

%
%

\begin{deff}
    The map $f$ has 
    \emph{stochastic stability with speed $\sigma$ 
        {for observables from $\bspace$} 
        with respect to $\seq{X_\eps}_{\eps\in\epsset}$ }
    if for every $\phi \in \bspace$ 
    there exists a constant $C(\phi)$ such that we have  
    \begin{gather*}
        \distmod{\phi}{\mu}{\mu_{\eps} } \leq
        C(\phi) \sigma(\eps).
    \end{gather*}
    
    We call $\sigma$ a \emph{speed of
    stochastic stability}.
\end{deff}

\begin{rem}
    If the system has this property
    it means that in the limit
    a typical pseudotrajectory approximates
    invariant measure of the initial system.
\end{rem}

\begin{rem}
    If $\mu$ and $\mu_\eps$
    have densities $\rho$ and $\rho_\eps$
    respectively with respect to Lebesgue 
    measure on $\mfd$ and
    the system has
    strong stochastic stability
    with speed $\sigma$ than it also has
    stochastic stability with speed $\sigma$
    for observables from $\bspacec$.

    Fix $\phi\in C(M)$ then
    \begin{gather*}
        \distmod{\phi}{\mu}{\mu_{\eps} } = 
        \vmod{\int \phi (\rho- \rho_\eps) d \Leb } \leq\\
        \leq \int \vmod{\phi} \vmod{\rho-\rho_\eps} d \Leb
        \leq \enbrace{ \sup_{x\in\mfd} \vmod{\phi(x)} }
        \sigma(\eps).
    \end{gather*}
\end{rem}

\begin{rem}
    If the Markov chains $\seq{X_\eps}$ are generated by
    random maps (see \cite{BONATTI_DIAZ_VIANA__BEYOND,JOST_RODRIG}
    for details), then the definition resembles the
    definition of inverse shadowing in \cite{PILINSCM}.
    However for this classical inverse shadowing it
    is shown that it can be quantitatively good 
    only for hyperbolic systems (see \cite{PILMELISP}).
\end{rem}



\begin{deff}
    The map $f$ has 
    \emph{\shadname shadowing with accuracy $\sigma$ 
        for observables from $\bspace$ 
        with respect to $\seq{X_\eps}_{\eps\in\epsset}$ }
        for every 
        $0<\eps<\eps_0$ for every $\phi\in \bspace$ 
        for every $N\in\Nb$
    there exists a constant $C(\phi)$
    and a set $\Bset{\phi}{N}$
    of realizations of $X_\eps$
    such that for every $\xbar = \seqinf{x}{k} \in \Bset{\phi}{N}$
       there exists a 
    point $p\in M$ such that the following holds for $n>N$:
    \begin{gather*}
        \distmod{\phi}{ S_n(\etraj{\xbar}) }{\mu_\eps }
        \leq C(\phi) \sigma(\eps), \nonumber \\
        \distmod{\phi}{ S_n(\etraj{p})}{\mu }
        \leq C(\phi) \sigma(\eps), \label{eq:obsshad1}\\
            \distmod{\phi}
            { S_n(\bar{x})}{ S_n( \etraj{p} ) } 
                \leq C(\phi) \sigma(\eps), \nonumber \label{eq:obsshad2}
    \end{gather*}
    and $\mu_\eps (\Bset{\phi}{N} ) \to 1$ as $N \to \infty$.

    We call $\sigma$ a \emph{\shadname shadowing 
    accuracy function}.
\end{deff}


\begin{rem}   \label{rem:class_close_impl_erg}
    If a point $x$ shadows a pseudotrajectory $\xbar=\seq{x_k}$
    in a usual sense, that is
    \begin{equation*}
        \dist(x_{n},f^n(x)) \leq \delta,\quad
        n\in\Nb,
    \end{equation*}
    then for every $C$-Lipschitz $\phi:M\to \Rb$ we have for
    every natural $n$ that
    \begin{equation*}
        \distmod{\phi}{S_n(\xbar)}{ S_n(\etraj{x}) }
    \leq C\delta.
    \end{equation*}
\end{rem}

Here is an explanation why does one need another 
notion similar to stochastic stability.
\begin{rem}
    It is easy to see that  for every $\phi\in C(M)$ and
    $n\in\Nb$ we can write 
    the following representations:
    \begin{gather}
         \differmod{\phi}{S_n(x) }{S_n(\xbar)} = \nonumber \\
         = 
         \differmod{\phi}{\mu}{\mu_\eps} + \nonumber \\
         \differmod{\phi}{S_n(x) }{\mu} +
         \differmod{\phi}{\mu_\eps}{S_n(\xbar)}   \label{eq:differmod_rep}
    \end{gather}
    and
    \begin{gather*}
         \differmod{\phi}{S_n(x)}{\mu } = \\
         = \differmod{\phi}{S_n(x) }{ S_n(\xbar) } +
         \differmod{\phi}{S_n(\xbar)}{\mu_\eps} +
         \differmod{\phi}{\mu_\eps}{\mu}  .
    \end{gather*}

    Then it is easy to deduce from stochastic stability with speed $\sigma$ 
    %
    that \emph{in the limit} $\mu_\eps$-almost every erroneous trajectory
    has the same statistics as $\mu$-almost every exact one:
    \begin{gather*}
        \varlimsup_{n\to\infty} \distmod{\phi}{S_n(x) }{S_n(\xbar)} \leq C(\phi) \sigma(\eps), \\
        \varlimsup_{n\to\infty} \distmod{\phi}{S_n(x) }{\mu} \leq C(\phi) \sigma(\eps)  .
    \end{gather*}
    A similar observation can be found in \cite{BLANK_EPS_TRAJ}.
    However what one really wants to have is that
    for many erroneous
    trajectories one can find an exact trajectory such that
    for \emph{the same number of iterations} the
    ergodic average for the exact one is close to $\mu$ and
    also close to the ergodic average for the erroneous one.

    Moreover since one usually pick erroneous trajectory at
    random it is highly desirable so that one could choose 
    a lower bound for the necessary number of iterations
    uniformly for those many erroneous trajectories. 

    And as there is no kind of ``speed of ergodic theorem'' statement
    in general, terms \eqref{eq:differmod_rep}
    for a fixed $n$ can fluctuate a lot,
    depending on $\xbar$ and $x$.
\end{rem}



%% file: defs_uniform.tex
\subsection{Uniform stability and shadowing}

Let $\seq{f_{\alpha}}_{\alpha \in \paramset}$,
where $A$ is a compact metric space,
be a family of maps from $M$ to itself. 
\begin{assume}
    Assume each $f_\alpha$ has a
    unique physical measure $\muparam{\alpha}$.
\end{assume}

Let be $\seq{ \Xparam{\eps}{\alpha} }_{\alpha\in \paramset, \eps\in \epsset}$
a family of Markov chains on $M$
such that every realisation of $X_{\alpha,\eps}$ is an
$\eps$-pseudotrajectory of $f_\alpha$.
\begin{assume}
    Assume for every $(\alpha,\eps)$ there is an ergodic
    stationary measure $\mustparam{\eps}{\alpha}$
    for $ \Xparam{\eps}{\alpha}$.
\end{assume}

We define here uniform versions
of the shadowing and stability notions,
requiring one accuracy function to
suit all the maps of the family.

\begin{deff}
    The family $\seq{f_{\alpha}}_{\alpha\in\paramset}$ has
    \emph{uniform} strong stochastic stability with speed $\sigma$ with 
    respect to $\seq{ \Xparam{\eps}{\alpha} }_{\eps\in\epsset, \alpha\in\paramset}$ 
    if for every $\alpha\in \paramset$ 
    the map 
    $f_{\alpha}$ has strong stochastic
    stability with speed $\sigma$ with 
    respect to $\seq{ \Xparam{\eps}{\alpha}
    }_{\eps\in\epsset}$.

    We call $\sigma$ a \emph{speed of uniform strong stochastic stability}.
\end{deff}

\begin{deff}
    The family $\seq{f_{\alpha}}_{\alpha\in\paramset}$ has
    \emph{uniform} stochastic stability
    with speed $\sigma$ 
    for observables from $\bspace$.
    with 
    respect to $\seq{ \Xparam{\eps}{\alpha} }_{\eps\in\epsset, \alpha\in\paramset}$ 
    if for every $\alpha\in \paramset$ 
    the map 
    $f_{\alpha}$ has stochastic stability
    with speed $\sigma$ with 
    for observables from $\bspace$
    respect to
    $\seq{ \Xparam{\eps}{\alpha} }_{\eps\in\epsset}$ 
        with the same  
        $C(\phi)$ for
        every $\phi\in \bspace$. 

    We call $\sigma$ a \emph{speed of uniform
    stochastic stability}.
\end{deff}

%

\begin{deff}
    The family $\seq{f_{\alpha}}_{\alpha\in\paramset}$
    is said to have \emph{uniform classical
        shadowing with accuracy $\sigma$
        with respect to
        $\seq{ \Xparam{\eps}{\alpha} }_{\eps\in\epsset, \alpha\in\paramset}$ } 
        if 
    for every $\alpha\in \paramset$ 
    the map
    $f_{\alpha}$ has classical shadowing with accuracy $\sigma$ with 
    respect to $\seq{ \Xparam{\eps}{\alpha} }_{\eps\in\epsset}$ 

    
    We call $\sigma$ a \emph{uniform classical
    shadowing accuracy function}.
\end{deff}

%

\begin{deff}
    The family $\seq{f_{\alpha}}_{\alpha\in\paramset} $
    has \emph{uniform} \shadname
    shadowing property with accuracy $\sigma$
    for observables from $\bspace$
    with respect to $\seq{ \Xparam{\eps}{\alpha}
        }_{\eps\in\epsset, \alpha\in\paramset}$ 
        if for every $\alpha \in \paramset$ 
        the map
        $f_{\alpha}$ has \shadname shadowing property
        for observables from $\bspace$
        with accuracy $\sigma$ 
        with respect to
        $\seq{ \Xparam{\eps}{\alpha} }_{\eps\in\epsset}$
        %
        with the same  
        $C(\phi)$ for
        every $\phi\in \bspace$. 

    We call $\sigma$ a \emph{uniform \shadname shadowing 
    accuracy function}.
\end{deff}


Here are several definitions we will use in Section 
\ref{sec:ex}.

Put $c=1/2$.
For $s \in \tentparamset $
denote by $f_s$ a tent map with slope $s$:
\begin{gather*}
    f_s(x) = \begin{cases}
        sx & x < c, \\
        s - sx & c < x < 1.
    \end{cases}
\end{gather*}

Fix an $s_0\in \tentparamset$ and
denote $g=f_{s_0}$.

\begin{deff}
    We call a
    continuous map $\hat{g}:[0,1] \to [0,1]$  
    \emph{piecewise
    expanding $C^r$ unimodal} for $r\geq 1$ if
 there exists $a\in (0,1)$ such that 
 $g|_{[0,a]}$ is strictly increasing and extends to a $C^r$ map
 in the neighborhood of $[0,a]$ and
 $g|_{[a,1]}$ is strictly decreasing and extends to a $C^r$ map
 in the neighborhood of $[a,1]$.
\end{deff}

\begin{deff}
Let $r\geq r_0\geq 2$ be integers. 
A \emph{$C^{r_0,r}$-perturbation} of $g$
is a family of piecewise expanding
$C^r$ unimodal maps $g_t:[0,1]\to[0,1]$, $t\in [-1,1]$
with $f_0 = f$ and satisfying the following properties:
there exists neighborhoods $I_{1},I_2$ of
$[0,c]$ and $[c,1]$
respectively so that the $C^r$ norm of the extension of
$g_t|_{I_i},\ i=1,2$ is uniformly bounded for small $\vmod{t}$
and so that
\begin{equation*}
    \normban{ (g-g_t)|_{I_i}}_{C^{r-1}} \leq Ct,\quad
    i=1,2
\end{equation*}
for some $C>0$.
The map $(x,t)\to g_t(x)$ extends to a $C^{r_0}$
function on a neighbourhood of $(I_1 \cup I_2) \times \{0\}$.
\end{deff}

\begin{deff}
Let $r\geq r_0\geq 2$ be integers. 
A $C^{r_0,r}$-perturbation of $g$ is \emph{tangent to
the topological class of $g$} if there exists a
$C^{2,2}$-perturbation $\tilde{g}_t$ of $f$
such that
\begin{equation*}
    \sup_{x\in M}\vmod{g_t(x) - \tilde{g}_t(x)} = \obig{t^2}
\end{equation*}
and homeomorphisms $h_t$ with $h(c)=c$ such that 
$\tilde{g}_t = h_t \circ g_t \circ h^{-1}_t$.
    
\end{deff}

%% file: mainres.tex
\section{Main Results}

For the sake of brevity we will not 
mention the perturbation with respect to which
\shadname shadowing or stochastic stability holds,
always meaning $\seq{X_\eps}_{\eps\in\epsset}$
that we fixed before.

First we state some general theorems.

\begin{maintheorem}\label{mainth:stshad_impl_ststab}
    If the map $f$ has \shadname shadowing 
    with accuracy $\sigma$ then
    for observables from $\bspace$ 
    it has stochastic stability with speed $\sigma$
    for observables from $\bspace$.
\end{maintheorem}

Denote by $\fspacel$ the set of Lipschitz functions from
$M$ to $\Rb$.

\begin{maintheorem} 
        \label{mainth:ststab_impl_stshad}
    If the map $f$ has stochastic stability with speed
    $\gamma$ 
    for observables from $\fspacel$
    and 
    classical shadowing with accuracy $\sigma$ and
    exponential decay of correlations 
    for observables from $\fspacel$
    then
    it has \shadname shadowing with accuracy
    $\max(\sigma,\gamma)$ for observables
    from $\fspacel$.
\end{maintheorem} 

\begin{maintheorem} \label{mainth:exp_decay_specif_impl_stochshad}
    If $f$ is continuous and has stochastic stability 
    with speed $\sigma$
    for observables from $\bspacec$ 
    and the following
    conditions are satisfied:
    \begin{itemize}
        \item $f$ has specification property,
        \item $\mu$ is not a finite
            sum of $\delta$ measures,
        \item   the stationary measure $\mu_\eps$ has
            exponential decay of correlations
            for observables from $\bspacec$ 
    \end{itemize}
    then 
    it has \shadname shadowing  with accuracy
    $\sigma$
    for observables from $\bspacec$.
\end{maintheorem} 

As specification property is known to hold for 
certain systems (see \cite{ SIGMUND_SPECIF}) and 
for interval maps topological mixing
implies specification (see \cite{BLOKH})
and strong stochastic stability is proved for
many systems (see \cite{VIANA_ONLINE, BLANK_KELLER})
we can state the following result:

\begin{maincor}
    The following systems 
    have \shadname shadowing 
    for observables from $\bspacec$,
    with accuracy equal to 
    the speed 
    of stochastic stability:
    \begin{itemize}
        \item uniformly hyperbolic attractors,
        \item expanding maps of closed manifolds,
        \item topologically mixing piecewise expanding maps of the interval,
        \item topologically mixing smooth unimodal maps of the interval
            with exponential decay of correlations (see \cite{VIANA_ONLINE}
            for exact conditions to get exponential decay).
    \end{itemize}
\end{maincor}

\begin{rem}
    Classes of perturbations in the previous theorem for which
    stochastic stability is proved are somewhat different.

We also state uniform versions of some theorems above:

\begin{maintheorem}
    Let family 
 $\seq{f_{\alpha}}_{\alpha \in \paramset}$
    have uniform \shadname shadowing 
    with accuracy $\sigma$ 
    for observables from $\bspace$ 
    then
    it has uniform stochastic
    stability with speed $\sigma$
    for observables from $\bspace$.
\end{maintheorem}

\begin{maintheorem} 
    Let family 
 $\seq{f_{\alpha}}_{\alpha \in \paramset}$
    have 
    uniform 
    stochastic stability with speed
    $\gamma$
    for observables from $\fspacel$,
    uniform classical shadowing with accuracy $\sigma$ 
    and 
    for each $\alpha\in\paramset$ the map $f_\alpha$ has
    exponential decay of correlations 
    for observables from $\fspacel$
    then
    the family 
 $\seq{f_{\alpha}}_{\alpha \in \paramset}$
    has uniform \shadname shadowing with accuracy
    $\max(\sigma,\gamma)$ for 
    observables from $\fspacel$.
\end{maintheorem} 

\begin{maintheorem} 
    Let family 
 $\seq{f_{\alpha}}_{\alpha \in \paramset}$
 have uniform stochastic stability 
    with speed $\sigma$
    for observables from $\bspacec$ 
    and 
    for every $\alpha\in\paramset$
    the following
    conditions are satisfied:
    \begin{itemize}
        \item $f_\alpha$ is continuous,
        \item $f_\alpha$ has specification property,
        \item $\muparam{\alpha}$ is not a finite
            sum of $\delta$ measures,
        \item   the stationary measure
            $\mustparam{\alpha}{\eps}$ has
            exponential decay of correlations
            for observables from $\bspacec$ 
    \end{itemize}
    then 
 $\seq{f_{\alpha}}_{\alpha \in \paramset}$
    has uniform \shadname shadowing  with accuracy
    $\sigma$
    for observables from $\bspacec$.
\end{maintheorem} 

Finally we state results concerning the example.

\begin{maintheorem}\label{mainth:lipshad_tent}
For a full measure (in $\tentparamset$) set of parameters
$s$
     for any $\alpha,L>0$ 
   the map $f_s$ does not have
   standard shadowing with accuracy
   $\sigma(\eps) = L\eps^\alpha$.

   Nevertheless for every $s\in \tentparamset$
the map $f_s$ has 
stochastic shadowing 
    with speed $\sigma(\eps) = \eps$
    for observables from $C(M)$
    for the family of uniform perturbations.
\end{maintheorem}
\begin{rem}
    However despite the same accuracy function
    suits all the $f_s$ the stochastic shadowing may be
    nonuniform since we can not guarantee 
    uniformity of constants $C(\phi)$.
\end{rem}

\begin{maintheorem} \label{mainth:unif_tangent}
    Let $s\in\tentparamset$ be a nonperiodic parameter.
    If $\seq{g_t}$ is a $C^{2,2}$-perturbation of
    $f_s$, tangent to its topological class,
    then there exists $t_0>0$ such that 
    the family $\seq{g_t}_{\vmod{t}<t_0}$
    has uniform 
    stochastic shadowing 
    with speed $\sigma(\eps) = \eps$ with
    respect to $\seq{X_\eps}$
    for observables from $C(M)$ 
    for the family of uniform perturbations.
\end{maintheorem}

%% file: proof.tex
\section{Proofs}

\begin{theorem}[Theorem \ref{mainth:stshad_impl_ststab}]
    If the map $f$ has \shadname shadowing 
    with accuracy $\sigma$ 
    for observables from $\bspace$ 
    then
    it has stochastic stability with speed $\sigma$
    for observables from $\bspace$.
\end{theorem}

\begin{proof}
    Let $\sigma$ be an \shadname 
    shadowing accuracy function for $f$.
    Let $\eps_0$ be a number from the definition of \shadname
    shadowing.
    Fix $\eps < \eps_0$ 
    and $\phi\in \bspace$.
    There exists a set $\Bset{\phi}{n}$ from 
    the definition of \shadname shadowing and 
    a natural number $N$ such that 
    $\Bset{\phi}{N} \neq \emptyset$.

    Fix a natural $n>N$.  
    Fix an $\eps$-pseudotrajectory $\xbar\in\Bset{\phi}{N}$. 
    Then there exists a point $p$  such that
    \begin{gather*}
        \distmod{\phi}{\mu}{\mu_\eps} \leq 
    \distmod{\phi}{\mu}{ S_n( \etraj{x} )} + 
    \distmod{\phi}{ S_n(\etraj{x})}{ S_n( \xbar) } + \\
    +
    \distmod{\phi}{ S_n(\xbar)}{ \mu_\eps} \leq 3C(\phi) \sigma(\eps).
    \end{gather*}
    Thus we have
    stochastic stability with the desired speed.

\end{proof}

\begin{corollary}
    Lower bounds on the speed of stochastic stability would imply lower bounds
    on the accuracy of \shadname shadowing.
\end{corollary}


\begin{theorem}
    If the map $f$ has stochastic stability with speed
    $\gamma$ 
    for observables from $\fspacel$
    and 
    classical shadowing with accuracy $\sigma$
    and property A
    for observables from $\fspacel$
    then
    it has \shadname shadowing with accuracy
    $\max(\sigma,\gamma)$ for 
    observables from $\fspacel$.
\end{theorem} 
\begin{proof}
    Fix $\eps < \eps_0$ and 
    a $C$-Lipschitz function $\phi:M\to \Rb$.
    Set $\Bset{\phi}{N} = \Aset{\phi}{\sigma(\eps)}{N}$
    and fix $N$ large enough
    so that
    $\Bset{\phi}{N} \neq \emptyset$.

    Fix $\xbar\in \Bset{\phi}{N}$.
    By Remark \ref{rem:class_close_impl_erg} there is a point $x$
    such that 
    \begin{equation*}
        \distmod{\phi}{S_n(\xbar)}{ S_n(\etraj{x}) }
    < C\sigma(\eps) ,\quad n\in\Nb.
    \end{equation*}
    We have the following estimate for $n>N$
    \begin{gather*}
        \distmod{\phi}{\mu}{ S_n(\etraj{x} )} \leq 
        \distmod{\phi}{\mu}{\mu_\eps} +
        \distmod{\phi}{\mu_\eps}{ S_n(\xbar)} + \\
        + \distmod{\phi}{ S_n(\xbar)}{ S_n(\etraj{x}) }
        \leq
        \gamma(\eps) + (1+C)\sigma(\eps).
    \end{gather*}
\end{proof}

The following readily follows from Remark \ref{rem:expdecay_impl_propA}.
\begin{corollary}[Theorem \ref{mainth:ststab_impl_stshad}]
    If the map $f$ has stochastic stability with speed
    $\gamma$ 
    for observables from $\fspacel$
    and 
    classical shadowing with accuracy $\sigma$ 
    and
    exponential decay of correlations 
    for observables from $\fspacel$
    then
    it has \shadname shadowing with accuracy
    $\max(\sigma,\gamma)$ for observables
    from $\fspacel$.
\end{corollary} 




\begin{rem}
    It should be noticed that 
    a conceptually similar statement 
    was proposed in \cite{BLANK_EPS_TRAJ}.
\end{rem}

\begin{rem}
    The condition of presence of classical shadowing
    is not necessary.
    If we take $M = S^1$ and $f$ to 
    be an irrational rotation then
    it is easy to see that the system has 
    stochastic shadowing with respect
    to family of uniform perturbations while it 
    does not have
    classical shadowing with respect to the same 
    family of perturbations.
\end{rem}

\begin{rem}
    One can try to go further in quantifying the property
    of stochastic shadowing, estimating
    dependence of $N$ and $\mu_\eps(\Bset{\phi}{N})$
    on $\eps$ to distinguish between
    rotations and hyperbolic systems.
\end{rem}

\begin{rem}
    A simple note to make is 
    that if we drop the requirement that the measure
    $\mu$ is ergodic and physical and take $f=Id$ then 
    it is easy to see that the system does not
    have stochastic shadowing while it has
    stochastic stability (with any speed for a wide
    choice of perturbations).
\end{rem}



\begin{theorem}[Theorem \ref{mainth:exp_decay_specif_impl_stochshad}]
    If $f$ is continuous and has stochastic stability 
    with speed $\sigma$
    for observables from $\bspacec$ 
    and the following
    conditions are satisfied:
    \begin{itemize}
        \item $f$ has specification property,
        \item $\mu$ is not a finite
            sum of $\delta$ measures,
        \item   the stationary measure $\mu_\eps$ has
            exponential decay of correlations
            for observables from $\bspacec$
    \end{itemize}
    then 
    it has \shadname shadowing  with accuracy
    $\sigma$
    for observables from $\bspacec$.
\end{theorem}

\begin{proof}
    Fix $\phi\in \bspacec$
    and $0<\eps<\eps_0$.

    By Remark \ref{rem:expdecay_impl_propA} 
    the system has property A.
    Set $\Bset{\phi}{N} = \Aset{\phi}{\sigma(\eps)}{N}$
    and fix $N$ large enough so that $\Bset{\phi}{N} \neq \emptyset$
    and  $2/N < \sigma(\eps)$.
    Fix $\xbar\in \Bset{\phi}{N}$.

    Presence of specification implies 
    that periodic measures (measures, supported on periodic orbits of $f$) 
    are weak* dense among all invariant
    probability measures for $f$ (see \cite{SIGMUND_SPECIF}).

    Since $\mu$ is not a finite sum of $\delta$-measures,
    to approximate it by a periodic measure sufficiently good, the period should
    be large enough.
    Thus as $\phi$ is continuous,
    there exists a periodic point $p_n$  of period $n > N$
    such that 
    \begin{equation*}
        \distmod{\phi}{S_n( \etraj{p_n})}{ \mu } < \sigma(\eps).
    \end{equation*}

    Consider any integer $r = nk+m > N$ where $0 <m < n$.
    Note that
    \begin{gather*}
        Y = \int \phi d S_{r} (p_n) - \int \phi d S_{nk} (p_n)   =  \\
         =\sum_{i= 0}^{nk-1} \enbrace{ \frac{1}{nk} + \frac{1}{nk+m}} \phi(f^i(x)) -
        \sum_{i=0}^{m} \frac{1}{nk+m} \phi(f^i(x)) = \\
        = \frac{1}{nk+m} \enbrace{ 
            \sum_{j=0}^{m} \sum_{i=0}^{nk-1} \frac{\phi(f^i(x))}{nk} -  
            \sum_{j=0}^{m} \sum_{i=0}^{nk-1} \frac{\phi(f^j(x))}{nk} 
        }.
    \end{gather*}

    Then
    \begin{equation*}
        \vmod{Y}  \leq
         \frac{2(m+1)}{nk+m}
         \sup_{x\in\mfd} \vmod{\phi(x)}  \leq
         \frac{2}{k}   < \frac{2}{N} < \sigma(\eps)       .
    \end{equation*}

    Therefore we have 
    \begin{equation*}
        \distmod{\phi}{S_r(p_n)}{\mu}
        < \sigma(\eps) .
    \end{equation*}

    This implies that
    \begin{gather*}
        \distmod{\phi}{S_{r}(\xbar)}{S_{r}(\etraj{p_n} )} \leq \\
        \leq      
        \distmod{\phi}{S_{r}(\xbar)}{\mu_\eps} +
        \distmod{\phi}{\mu_\eps}{ \mu} 
        +    \\   
        +         
        \distmod{\phi}{\mu}{ S_{r}(\etraj{p_n} ) }  
        %
        \leq  3 \sigma(\eps)   .
    \end{gather*}

\end{proof}

\begin{rem}
    In fact, a summable decay of correlations
    is probably enough to get the same conclusion.
\end{rem}

\begin{rem}
    There are properties weaker than the specification property
    that can guarantee density of periodic measures
    among all invariant ones (see \cite{GELFERT_KWIET, KWIET_STUD}).
\end{rem}


\input{proof_uniform}

\end{rem}

Despite there is a vast literature
on stochastic stability
(see references in \cite{VIANA_ONLINE} and more recent in
\cite{SHEN_STRIEN_NEUTRAL_STSTAB, SHEN_UNIMOD_STSTAB}),
there are not many results giving exact form of
its speed. 
See \cite{BALADI_YOUNG} for a result about
expanding maps of a circle.
There exist some numerical studies of speed 
of stochastic stability in different situations
in \cite{LIN_NUM_SPEED}.

However these speeds can sometimes be
obtained rather easily
by direct applications of results of
Keller-Liverani \cite{KELLER_LIVERANI}.
%
%
We give an example when it happens in Section \ref{sec:ex}.




%% file: proof_uniform.tex

The proofs of uniform versions of the above theorems are 
just repetitions of proofs of the non-uniform versions.





%% file: ex_tent.tex
\section{Example} \label{sec:ex}

It was shown in \cite{YORKE_TENT_SHAD} that
for almost all parameters $s$ tent maps $f_s$ have shadowing property.
However we are going to show that it is by no means controllable, i.e.
not Lipschitz
and not uniform.
Afterwards we show that
nevertheless tent maps have stochastic shadowing
with linear accuracy function.

Let $M = [0,1]$.
Set
\begin{gather*}
    E_{m,\delta} = \setdef{f^m(y)}{ \vmod{f^i(y) - f^i(c)} \leq
\delta, \quad 0\leq i\leq M}, \\
    m(\delta) = \inf \setdef{m\in\Nb}{ c\in E_{m,\delta} }.
\end{gather*}
Here we mean that $\inf(\emptyset)=\infty$.

We say that a parameter $s\in\tentparamset$ is $N$-periodic 
if the critical point $c$ is periodic for $f_s$
with (minimal) period equal to $N$.

For $a,b\in \mfd$
denote 
\begin{equation*}
    \abr{a,b} = \begin{cases}
        [a,b], & a<b;\\
        [b,a], & a>b .
    \end{cases}
\end{equation*}

For an $N$-periodic $s$ denote
\begin{equation*}
    \xi_s = \min_{ 0 < k < N} { \vmod{c - f_s^k(c) }}  
\end{equation*}
and for $\eps,\delta>0$ set
\begin{equation*}
    n_s(\delta,\eps) = \min \setdef{ n \in \Nb}{ c \in 
    \abr{ f^n(f(c)-\delta), f^n(f(c)+\eps) } }.
\end{equation*}

\begin{lm}
    If $s$ is $N$-periodic, then for every 
    $\delta,\eps < \xi_s s^{-N} /2$ we have
    \begin{gather*}
        n_s(\delta,\eps) = N-1 ,\\
        m(\delta) = N-1 .
    \end{gather*}
\end{lm}

\begin{proof}
    Denote
    \begin{equation*}
        I_k = \abr{f^k(f(c)-\delta), f^k(f(c)+\eps) } .
    \end{equation*}
    We have $c_{k+1} \in I_k$
    while  $c\notin I_k$.
Due to the assumptions on $\delta$ and $\eps$
    we have $\dist(I_k, c) > 0$ as $k<N$.
    It means that $c \notin I_k$ until $k=N$.  
\end{proof}

The following theorem is proved in \cite{YORKE_TENT_SHAD}. 
\begin{theorem}
    Let $\shadparamset$ be a set of parameters from $\tentparamset$ such that
    for every $\delta>0$ the number $m(\delta)$ is finite. 

    Then $\shadparamset$ has full measure in $\tentparamset$ and for every
    $s\in \shadparamset$ for 
    \begin{equation*}
        \eps < \delta (s-1) (m(\delta)+1)^{-1} s^{-m(\delta)-1}
    \end{equation*}
    every $\eps$-pseudotrajectory can be 
    \begin{equation*}
        \enbrace{(s-1)^{-1} + s^4} (m(\delta)+1)
        s^{m(\delta)+1} \eps
    \end{equation*}
    -shadowed.
    
    In particular for every $N$-periodic parameter
    $s_N$ the map
    $f_{s_N}$ has standard shadowing with accuracy 
    \begin{equation*}
         \sigma(\eps) =  \enbrace{(s-1)^{-1} + s^4}
         (N+1) s^{N+1} \eps . 
    \end{equation*}
\end{theorem}


Denote 
\begin{equation*}
    c_k = f^k(c).
\end{equation*}

We need a following classical result
in one-dimensional dynamics (for the proof see,
for example \cite{YORKE_TENT_SHAD}):
\begin{lm}               
   The set of parameters $s$ such that there exists
   $N$ such that $s$ is $N$-periodic, is dense in
   $[\sqrt{2},2]$ but has zero measure.
\end{lm}

\begin{lm}[part of Theorem \ref{mainth:lipshad_tent}]
For a full measure (in $\tentparamset$) set of parameters $s$
     for any $\alpha,L>0$ 
   the map $f_s$ does not have
   standard shadowing with accuracy
   $\sigma(\eps) = L\eps^\alpha$.
\end{lm}
\begin{proof}
    Consider  
    \begin{gather*}
        x_0 = c,\\
        x_1 = f(c) + \eps,\\
        x_k = f^{k-1}(x_1), k\geq 2.
    \end{gather*}

    Suppose $\xbar$ can be $\omega$-shadowed by a point $y$. Then consider
    $x_{\eps} = f^2(y)$. We know that $x_{\eps}$ is not equal to $x_2$, otherwise
    $f(y) > f(c)$ which can not hold for an exact trajectory.

    \begin{sublm}
        We prove that 
         \begin{equation*}
             \sigma(\eps)
            \geq s^{n(\eps)-1} \eps,
         \end{equation*}
         where
     \begin{equation*}
         n(\eps) = \min \setdef{ n \in \Nb}
         { c \in \Int(\abr{ f^n(x_\eps), f^n(x_2) }) }.
         \end{equation*}
    \end{sublm}
    \begin{proof}[Proof of sublemma]


        It is easy to see that the distance between 
        $f^k(x_{\eps})$ and $f^k(x_2)$ grows
        exponentially while $k < n(\eps)$. 

    Thus we can estimate from below the accuracy of shadowing: 
        \begin{equation*}
            \sigma(\eps) \geq \vmod{f^{n(\eps)-1}(x_{\eps})  - f^{n(\eps)-1}(x_2) } = s^{n(\eps)-1}
            \vmod{x_\eps - x_2} \geq s^{n(\eps)-1} \eps.
        \end{equation*}
    \end{proof}

        If $\sigma(\eps) = \eps^{\alpha}$ for $0<\alpha<1$ then
        \begin{equation*}
            \eps \leq s^{\frac{1-n(\eps)}{1-\alpha} }
        \end{equation*}
        which implies
        \begin{equation*}
            \sigma(\eps) \leq s^{\alpha\frac{1-n(\eps)}{1-\alpha} }
        \end{equation*}
    
    If $c$ is not periodic then it is easy to see that
        $n$ goes to infinity as $\eps$ goes to $0$.
      %
      %
    By \cite{CRIT_TYP_BRUIN} (or, more generally
    \cite{CRIT_TYP_SCHNELL})
    the set of parameters with
    non-periodic $c$ has full measure in $\tentparamset$.

\end{proof}

\begin{lm}
    Let $U$ be an open subset of $\tentparamset$. Then
    \begin{equation*}
        \inf_{s\in U} \sup_{\delta,\eps>0}
        n_s(\delta,\eps) = \infty.
    \end{equation*}
\end{lm}
\begin{proof}
    Fix open $U\subset \tentparamset$.
    Then for every $N\in\Nb$ there exists
     an $n>N$-periodic
    parameter $s\in U$. 
    Since we know that 
    \begin{equation*}
        \sup_{\delta,\eps>0} n_s(\delta,\eps)  = n,
    \end{equation*}
    the statement of the lemma easily follows.
\end{proof}

For an open set $U\subset \tentparamset$
denote $U_{\mathrm{shad}}$ a subset parameters $s$ of $U$ 
such $f_s$ has standard shadowing (for some
accuracy function).
\begin{corollary}
    For any open $U \subset \tentparamset$
    we cannot choose a single constant $L$ 
    such that for every parameter
    $s\in U_{\mathrm{shad}}$
    the map
    $f_s$ has standard shadowing with accuracy 
    $\sigma(\eps)=L\eps$.
\end{corollary}

\input{erg_PF}
\begin{corollary}[Theorem \ref{mainth:unif_tangent}]
    There exists $t_0>0$ such that 
    the family $\seq{g_t}_{\vmod{t}<t_0}$
    has uniform 
    stochastic shadowing 
    with speed $\sigma(\eps) = \eps$ 
    for observables from $C(M)$ 
    with repsect to the family of uniform
    perturbations.
\end{corollary}

If we do not care about uniformity, it is possible to
prove an analog of
Theorem \ref{thm:pertLYUnif} (see \cite{BLANK_KELLER})
even for $g = f_s$ with a periodic parameter $s$. 
Therefore we have the following theorem:
\begin{theorem}
    For every $s\in [0,1]$
the map $f_s$ has 
strong stochastic stability
with speed $\sigma(\eps) = L\eps$ for some $L>0$
with repsect to the family of uniform
perturbations.
\end{theorem}

\begin{corollary}[part of Theorem \ref{mainth:lipshad_tent}]
    For every $s\in \tentparamset$
the map $f_s$ has 
stochastic shadowing 
    with speed $\sigma(\eps) = \eps$ 
    for observables from $C(M)$
    with repsect to the family of uniform
    perturbations.
\end{corollary}

%% file: erg_PF.tex
If there is some regularity of dependence of 
transfer operator spectral properties on the
parameter (usually it is studied in the context
``statistical stability'', i.e. continuous dependence of the
physical measure on $f$, or its quantitative versions
like ``linear response'')
then it is possible to get uniform stochastic stability.
We illustrate it by an example.

Fix a nonperiodic parameter $s\in\tentparamset$.
Denote $g=f_s$.

Let $\seq{g_t}$ be a $C^{2,2}$-perturbation of $g$,
tangent to its topological class.

For $\phi:[0,1] \to \Rb$ define variation of $\phi$ by
\begin{equation*}
    \var_{[0,1]}(\phi) = \sup \setdef{ \vmod{
        \sum \phi(x_{k+1}) - \phi(x_k) } }{ n\geq 1,\ 
        0\leq x_0<\ldots< x_n \leq b}  .
\end{equation*}
We consider Banach space 
$(\bspace,\normb{\cdot}) =
(BV,\normBV{\cdot})$ where
\begin{equation*}
    \normBV{\phi} = \var_{[0,1]}(\phi) + \normLone{\phi}
\end{equation*}
and $BV$ is a set of functions $\phi$ from $\Lone$ 
such that $\normBV{\phi} < \infty$.

For $\phi\in BV$ define 
\begin{equation*}
    L_{t,0}\phi(x) = \sum_{g_t(y)=x} \frac{\phi(y)}
    {\vmod{g_t'(y)}} .
\end{equation*}

Let  
$\seq{ \Xparam{\eps}{t} }_{t\in [-1,1],\eps\in \epsset}$ 
            be a family of uniform perturbations
            of $g_s$. 

            Note that in our case of $M=[0,1]$ 
            as $g_s([0,1])$ is strictly inside of $(0,1)$,
            we can
            write for every natural $n$
\begin{equation*}
 \enbrace{\frac{d}{\Leb} P^n_\eps(x,\cdot)}(y) =
    \theta_\eps(y) = \begin{cases}
        \frac{1}{ \Leb(B_\eps(0)) },& \vmod{y} \leq \eps; \\
        0, & \hbox{otherwise}.
    \end{cases}
\end{equation*}

Define for $0<\eps<\eps_0$
\begin{equation*}
    L_{t,\eps} \phi(x) = \int (L_{t,0}\phi)(x-\omega)
        \theta_\eps(\omega)d\omega .
\end{equation*}



Here is a particular case of the abstract setting
introduced in \cite{KELLER_LIVERANI}:

For a bounded linear operator 
$Q:\bspace\to \bspace$ 
denote 
\begin{equation*}
    \normopKL{ Q } = \sup\setdef{ \normLone{Qf} }
    { f\in \bspace,\ \normb{f}\leq 1}.
\end{equation*}
Consider a family $\seq{P_\eps}_{\eps\geq 0}$ of bounded
linear operators on $(\bspace,\normb{\cdot})$
with the following properties:
\begin{enumerate}
    \item  \label{enum:cond1}   there are constants $C_1,L$ such that for all $\eps\geq 0$

        \begin{equation*}
            \normLone{P^n_\eps} \leq C_1 L^n,\quad \forall n\in\Nb;
        \end{equation*}
    \item  \label{enum:cond2} 
        there are constants $C_2,C_3$ and $\alpha\in(0,1),\ \alpha<L$
        such that for all 
        $\eps \geq 0$
        \begin{equation*}
            \normb{P_\eps^n \phi} \leq C_2 \alpha^n \normb{\phi} + C_3 L^n \normLone{\phi},\quad
            \forall n\in\Nb,\ \forall \phi\in \bspace;
        \end{equation*}
    \item  \label{enum:cond3} 
        if $z\in \sigma(P_\eps)$ and $\vmod{z}>\alpha$
        then $z$ is not in the residual spectrum of $P_\eps$;
    \item  \label{enum:cond4}
        there is monotone upper-semicontinuous function $\tau:\Rplus\to\Rplus$
        such that $\tau(\eps)>0$ if $\eps>0$ and
        $\tau(\eps) \to 0$ as $\eps \to 0$ and
        \begin{equation*}
            \normopKL{ P_0 - P_\eps} \leq \tau(\eps).
        \end{equation*}
\end{enumerate}

It follows that $P_\eps$ has $1$ as simple eigenvalue
 for every $\eps\geq 0$. Denote the corresponding eigenfunction by $\chi_\eps$.

The next statement follows from \cite{KELLER_LIVERANI}:
 \begin{theorem} \label{thm:Keller-Liverani}
     If a family of operators $\seq{P_\eps}_{\eps\geq 0}$ satisfies
     conditions \ref{enum:cond1}-\ref{enum:cond4}
    then there exist constants $\eta = \eta(\alpha,L)>0$ and $\eps_0$  such that
    for every $\eps<\eps_0$
    \begin{equation*}
        \normLone{\chi_0 - \chi_\eps} \leq (\tau(\eps))^{\eta}.
    \end{equation*}
\end{theorem}

It is proved in \cite{BALADI_SMANIA_PE} that
the family $\seq{L_{t,0} }$ satisfies
 conditions \ref{enum:cond1}-\ref{enum:cond4}.   
%
%


Theorem \ref{thm:Keller-Liverani} can also
be applied to random peturbations.
%
%
To do this we need another important statement (Corollary, p. 327 from \cite{KELLER82}):


\begin{theorem} \label{thm:pertLYUnif}

    The family $L_\eps$ satisfies conditions \ref{enum:cond1}-\ref{enum:cond4}  
    for $\tau(\eps) 
    = L\eps$
    for some $L>0$.
\end{theorem}



\begin{corollary}
    There exists $t_0>0$ such that 
    the family $\seq{g_t}_{\vmod{t}<t_0}$
    has uniform strong stochastic
    stability with speed $\sigma(\eps) = L\eps$
    with repsect to the family of uniform
    perturbations.
\end{corollary}

%% file: ack.tex
\section{Acknowledgements}

This research is supported by the
Chebyshev Laboratory  (Department of Mathematics and
Mechanics, St. Petersburg State University)
[under RF Government grant 11.G34.31.0026];
 JSC "Gazprom Neft";
 RFBR [project 12-01-00275];
  St. Petersburg State University [thematic project 
  6.38.223.2014].


  The author is grateful 
 for discussions
  to
Marcelo Viana,
Sergey Tikhomirov,
Dominik Kwietniak,
Viviane Baladi,
Peidong Liu,
Michal Misiurewicz,
Sergei Pilyugin,
Sergey Kyzhevich,
Alexey Petrov,
Andrey Alpeev,
Fyodor Sandomirski,
Christian Rodrigues.

%% file: refs.tex
\bibliography{bibrefs_ergshad}{}
\bibliographystyle{plain}